\documentclass{amsart}
\usepackage{graphicx}
\usepackage{amsmath}
\usepackage{amsaddr}
\usepackage{amssymb}
\usepackage{pgf,tikz}
\usepackage{mathrsfs}
\usetikzlibrary{arrows}
\usepackage{amsthm}
\newtheorem{thm}{Theorem}
\newtheorem{defn}{Definition}
\newtheorem{corr}{Corollary}
\newtheorem{lemm}{Lemma}
\begin{document}
\definecolor{white}{rgb}{1,1,1}
\definecolor{blue}{rgb}{0,0,1}
\definecolor{uuuuuu}{rgb}{0,0,0}
\definecolor{qqqqff}{rgb}{0.,0.,1.}
\title{Identifying combinatorially symmetric Hidden Markov Models}
\author{Daniel Klaus Burgarth}
\address{Department of Mathematics, Aberystwyth University, Aberystwyth SY23 3BZ, UK}
\date{\today}
\begin{abstract}
We provide a sufficient criterion for the unique parameter identification of combinatorially symmetric Hidden Markov Models based on the structure of their transition matrix. If the observed states of the chain form a zero forcing set of the graph of the Markov model then it is uniquely identifiable and an explicit reconstruction method is given.\end{abstract}
\maketitle
\section{Introduction}
Hidden Markov Models are a major workhorse of speech recognition, gene analysis, computational finance and, more generally, machine learning~\cite{hmm}. The key idea is to consider a very simple Markov chain as an underlying model, but only allow for a restricted set of observations, while the remaining part of the dynamics remains hidden. Such models naturally lead to a vast variety of statistical questions, such as the estimation of the initial state, prediction of future events, or identification of the parameters of the underlying dynamics.

In terms of learning the parameters of a Hidden Markov Model, powerful algorithms such as the Baum-Welch algorithm~\cite{hmm} have been developed. These algorithms are statistical in nature and iteratively maximise likelihood functions. The aim of the present paper is to approach parameter identification in such models from a much more elementary perspective. Namely, we set the problem not as a statistical, but as a matrix problem, and ask the question: given a Hidden Markov Model with a certain structure, how many states do we need to observe such that its parameters are uniquely determined?

The type of inverse problem we consider is inspired by the work of Gladwell~\cite{gladwell}, and analogous to parameter identification in quantum physics~\cite{quantum}. In the context of linear Markov chains, questions similar in spirit were discussed in~\cite{gzyl}, were it was shown that certain `local' perturbations in infinite chains are uniquely identified by their impact on hitting times. We consider a simpler problem on finite chains, but for arbitrary topology of the Markov model.

We find that a sufficient criterion for the unique identifiability can be easily given in terms of the structure of the matrix of transition probabilities/rates of the discrete/continuous Markov chain, provided this matrix is combinatorially symmetric. Namely, if the set of states we observe is a zero forcing set~\cite{hogben} of the simple graph which describes the non-zero pattern of the matrix, then the Hidden Markov Model can be uniquely identified. This holds for discrete (Theorem 1) as well as for continuous (Theorem 2) chains. The proof is constructive and therefore provides an easy-to implement algorithm for the identification, although we do not consider the statistical properties of the problem and in particular the efficiency when only limited information is provided.

\section{Setup}
Throughout this article we consider homogeneous and  finite Markov chains with states $s=1,\ldots,S$. To begin, let us consider discrete-time Markov chains with states $X_n$  described by a right stochastic transition matrix $P$. In particular, the transition probability $\mathbb{P}(X_{n+1}=j|X_{n}=i)$ is given by the matrix element $(P)_{ij}\ge0$ of $P$, independently of $n$. $P$ is row stochastic with $\sum_{j=1}^S (P)_{ij}=1$. We will refer to properties of the Markov chain and its transition matrix $P$ interchangeably. A central property of matrices which we use in this article is combinatorial symmetry~\cite{comb}:
\begin{defn}[Combinatorial symmetry]
A matrix $M$ is combinatorially symmetric if and only if for all indices $i,j$, $(M)_{ij}\neq0$ implies $(M)_{ji}\neq0.$
\end{defn}
The advantage of the combinatorical symmetry is that we can describe the non-zero pattern of $M$ by an \emph{undirected} graph. For the purpose of this work, we ignore the diagonal entries of $M$ when constructing this graph such that it has no self-edges: 
\begin{defn}\label{graphofmatrix}
For a combinatorially symmetric matrix $M$ of order $S$, define a simple graph $\mathcal{G}(M)=(V,E)$ with $S$ vertices. The edges are given by the non-zero pattern of the off-diagonal elements of $M$, 
$\{i,j\}\in E \Leftrightarrow (M)_{ij}\neq 0, i<j$.
\end{defn}
Another central concept in this article is zero forcing: \begin{defn}[Zero Forcing \cite{zeroforcing1, zeroforcing2}]
Consider a subset $Z$ of vertices on a simple graph which are initially coloured blue. Let the colour propagate on the graph according to the following rules: 1) if a vertex is blue, it remains blue; and 2 ) if a blue vertex $i$ has exactly one non-blue neighbour $j$, this neighbour becomes blue. If eventually all vertices turn blue, we call the original set $Z$ a zero forcing set.
\end{defn}
For an example, consider Fig. 1. Furthermore, for any simple graph, we can consider a minimal zero forcing set $Z$ such that for all forcing sets $Z'$, $|Z|\le|Z'|$. The number of elements in that set is called the zero forcing number of the graph, $Z\mathcal(G)$. It is easy to see that a path has forcing number $1$, and a lattice of size $n\times m, n\le m$, has forcing number $n$. For sparse graphs, the zero forcing number tends to be much smaller than the order. 
\section{Main result}
We consider the case of an endlessly evolving finite Markov chain `observed' on a certain set $A$ of states. That is, at a given time $n$, we may check if the chain is on any state within $A$, and if it is, we know in which one. If it isn't, we only know that it is in one of the hidden states outside $A$, without knowing in which one. When discussing this setup we are not worried about statistical efficiency of estimation, but only about what could be estimated uniquely, in principle, provided enough samples are taken. This leads us to 
\begin{lemm}\label{observing}
If a set $A$ of states of a finite discrete time Markov chain can be observed, and $\mathcal{G}(P)$ is connected, then the corresponding dynamics $(P^n)_{ij}$ can be estimated arbitrarily well for all $i,j\in A$ and for all $n\in\mathbb{N}$.
\end{lemm}
\begin{proof}
Since the Markov chain is irreducible, by observing state $i$ we are guaranteed that eventually we well find $X_k=i$ for some $k$. We then wait for time $n$ and measure state $j$, thereby sampling the probability above. We repeat until we have a precise enough estimate.
\end{proof}
Setting $n=1$ in the Lemma implies that $(P)_{ij},\; i,j\in A$ is fixed by the observations in $A$. While this is rather obvious-- after all, the states are directly observed-- the interesting question is what can be inferred \emph{indirectly} about neighbouring nodes. The following Lemma provides a crucial step.

\begin{lemm}\label{inductionstep}
Consider a set of states $A$ of a  Markov chain. Assume $(P^n)_{ij}$ is known for all $i,j\in A$ and for all $n\in\mathbb{N}$. Let $k\in A$ and $\ell \in \overline{A}=V\backslash A$ be states such that $\ell$ is the only incoming and only  outgoing neighbor of $k$ in $\overline{A}$ (so $(P)_{\ell k}\neq 0, (P)_{k \ell}\neq 0$ and $j\in \overline{A},(P)_{j k}\neq 0\Rightarrow j=\ell$ and $
j\in \overline{A},(P)_{k j}\neq 0\Rightarrow j=\ell$). Then we can infer $(P^n)_{ij}$ for all $i,j\in A\cup \{\ell\}$ and for all $n\in\mathbb{N}$ uniquely.
\end{lemm}
\begin{proof}Because of row stochasticity and $\ell$ being the only element in $\overline{A}$ with $(P)_{kj}\neq 0$, we have \begin{equation}\label{norm}1=\sum_j (P)_{kj}=\sum_{j\in A}(P)_{kj}+(P)_{k\ell}\end{equation}The first term is known by assumption, so Equation~\ref{norm} determines $(P)_{k\ell}$. Furthermore, for any $i \in A$, we have through Chapman-Kolmogorov that for all $n$,
\begin{align}
(P^{n+1})_{ki}&=\sum_j (P)_{kj}(P^{n})_{ji}\\
&=\sum_{j\in A} (P)_{kj}(P^{n})_{ji}+(P)_{k \ell}(P^{n})_{\ell i},
\end{align}
where the left hand side and the first term of the right hand side are already determined, hence fixing $(P^n)_{\ell i}.$ Similarly, using that $(P)_{\ell k}$ is the only non-zero term connecting $k$ with $\overline{A}$, we consider for all $n$ and for all $i \in A$
\begin{align}
(P^{n+1})_{ik}&=\sum_j (P^n)_{ij}(P)_{jk}\\
&=\sum_{j\in A}(P^n)_{ij}(P)_{jk}+(P^n)_{i\ell} (P)_{\ell k}
\end{align}
providing us with $(P^n)_{i\ell}$. Finally, for all $n$,
\begin{align}
(P^{n+1})_{\ell k}&=\sum_j (P^n)_{\ell j}(P)_{jk}\\
&=\sum_{j\in A}(P^n)_{\ell j}(P)_{jk}+(P^n)_{\ell \ell} (P)_{\ell k},
\end{align}
where $(P^n)_{\ell \ell}$ is the only unknown. Concluding, we know $(P^n)_{i j}$ on $A\cup \{\ell\}$.
\end{proof}

\setcounter{thm}{0}
We are now ready to prove the main theorem.
\begin{thm}\label{discrete}
Consider combinatorially symmetric discrete-time Markov chain $P$ with connected graph $\mathcal{G}(P)$. If the observable set of vertices on $\mathcal{G}(P)$ is zero forcing, then the Markov chain is uniquely determined through the dynamics of the observed states.
\end{thm}
\begin{proof}
We apply Lemma \ref{observing} to the initial set $A$ and increase it inductively using zero forcing and combinatoric symmetry in Lemma \ref{inductionstep} until all vertices are included.
\end{proof}
\begin{corr}
The zero forcing number of a combinatorially symmetric discrete-time Markov chain $P$ with connected graph $\mathcal{G}(P)$ is an upper bound to the minimal number of observed states which uniquely determine it.
\end{corr}
\section{Discussion of requirements}
In order to obtain necessary and sufficient criteria for the inverse problem provided, let us develop another perspective. For a given simple graph, consider the combinatorially symmetric Markov chains which provide such graph. The observed $(P^n)_{ij},\; i,j\in A$ are polynomials in the unknown elements of $P$. For example, for a path on three vertices,
\begin{equation}P=\begin{pmatrix} P_{11} & P_{12} & 0\\
P_{21} & P_{22} & P_{23}\\
0 & P_{32} & P_{33}
\end{pmatrix},\end{equation}
we obtain
\begin{align}\label{poly}
(P)_{11}&= P_{11}\\
(P^2)_{11}&= P_{11}^2+P_{12} P_{21}\\
(P^3)_{11}&= P_{11}^3+2 P_{11} P_{12} P_{21} + P_{12}P_{21}P_{22}\\
&\cdots \nonumber \\
(P^n)_{11}&=\cdots. \nonumber
\end{align}
The question of unique identification of the Markov chain can then be rephrased into the unique solvability of the above system of polynomial equations (combined with the row normalisation providing further equations). For larger systems this quickly becomes intractable, but it allows us to discuss some of the conditions posed above. Firstly, we note that the positivity of the non-zero entries was never used in the previous sections. It is therefore interesting to ask if one can apply Theorem \ref{discrete} for matrices without row normalisation? The above example of the path gives an immediate counterexample, because $P_{21}$ and $P_{12}$ always occur as a product, in any order $n$, and cannot be separated without another equation.

Furthermore, an important assumption in the above is the combinatorial symmetry of the Markov chain. What happens if this is not the case? As an example, consider the directed cycle on $S$ vertices. The zero forcing condition can easily be extended to directed graphs giving the directed cycle a forcing number of $1$. Setting up the polynomials as above, we obtain
\begin{align}
(P)_{11}&=0\\
(P^2)_{11}&=0\\
&\cdots\\
(P^S)_{11}&=\prod_{i=1}^{S-1} P_{i\;i+1}\\
(P^{S+1})_{11}&=\prod_{i=1}^{S-1} P_{i\;i+1}\sum_{i=1}^{S} P_{ii}\\
&\cdots.\nonumber
\end{align}
It is clear by symmetry that $P$ is not uniquely determined. We conjecture that the minimum number of vertices to be observed is actually $S-1$. Therefore, an analogous version of Theorem~\ref{discrete}, replacing combinatorial symmetry with a directed graph and generalising zero forcing along those lines, does not hold.

\begin{figure}
\centering
\begin{tikzpicture}[line cap=round,line join=round,>=triangle 45,x=0.2cm,y=0.2cm]
\clip(-17.,-14.) rectangle (22.,25.);
\draw [line width=1.2pt] (-5.268055785962234,6.373759434988376)-- (2.82211415778724,12.251611957913104);
\draw [line width=1.2pt] (2.82211415778724,12.251611957913104)-- (10.912284101536711,6.373759434988372);
\draw [line width=1.2pt] (10.912284101536711,6.373759434988372)-- (7.822114157787238,-3.1368057279631607);
\draw [line width=1.2pt] (-2.1778858422127616,-3.1368057279631607)-- (7.822114157787238,-3.1368057279631607);
\draw [line width=1.2pt] (-2.1778858422127616,-3.1368057279631607)-- (-5.268055785962234,6.373759434988376);
\draw [line width=1.2pt] (2.82211415778724,12.251611957913104)-- (2.7912554280289386,22.25156434473964);
\draw [line width=1.2pt] (10.912284101536711,6.373759434988372)-- (20.421268586471918,9.468789978498645);
\draw [line width=1.2pt] (7.822114157787238,-3.1368057279631607)-- (13.704801790334777,-11.22346055043903);
\draw [line width=1.2pt] (-2.1778858422127616,-3.1368057279631607)-- (-8.082785099457986,-11.207255824567735);
\draw [line width=1.2pt] (-5.268055785962234,6.373759434988376)-- (-14.77417089703853,9.477591823337802);
\begin{scriptsize}
\draw [fill=white] (-2.1778858422127616,-3.1368057279631607) circle (2.5pt);
\draw[color=uuuuuu] (-1.822222242469325,-1.5637936121808043) node {7};
\draw [fill=white] (7.822114157787238,-3.1368057279631607) circle (2.5pt);
\draw[color=uuuuuu] (9.358096886124487,-2.3260880982212915) node {8};
\draw [fill=white] (10.912284101536711,6.373759434988372) circle (2.5pt);
\draw[color=uuuuuu] (11.22148340755679,7.92253776965637) node {4};
\draw [fill=white] (2.82211415778724,12.251611957913104) circle (2.5pt);
\draw[color=uuuuuu] (3.5138391598140855,13.427997946615445) node {3};
\draw [fill=white] (-5.268055785962234,6.373759434988376) circle (2.5pt);
\draw[color=uuuuuu] (-4.956099573969106,7.92253776965637) node {2};
\draw [fill=white] (2.7912554280289386,22.25156434473964) circle (2.5pt);
\draw[color=uuuuuu] (3.0903422231249262,23.84602258916877) node {6};
\draw [fill=white] (20.421268586471918,9.468789978498645) circle (2.5pt);
\draw[color=uuuuuu] (20.792514176731796,11.056415101156151) node {5};
\draw [fill=blue] (13.704801790334777,-11.22346055043903) circle (2.5pt);
\draw[color=uuuuuu] (14.016563189705243,-9.694934796612669) node {10};
\draw [fill=blue] (-8.082785099457986,-11.207255824567735) circle (2.5pt);
\draw[color=uuuuuu] (-7.751179356117559,-9.610235409274837) node {9};
\draw [fill=white] (-14.77417089703853,9.477591823337802) circle (2.5pt);
\draw[color=uuuuuu] (-14.44243095580628,11.056415101156151) node {1};
\end{scriptsize}
\end{tikzpicture}
\caption{The penta-sun $H_5$~\cite{hogben}. The zero forcing number is $3$, and an example of a forcing set is $\{9,10,1\}$. If we only observe the smaller set $\{9,10\}$ (blue), we can use Lemma~\ref{inductionstep} to obtain $(P^{n})_{ik}$ for $i,k \in \{2,4,7,8,9,10\}$ and all $n$. }
\end{figure}
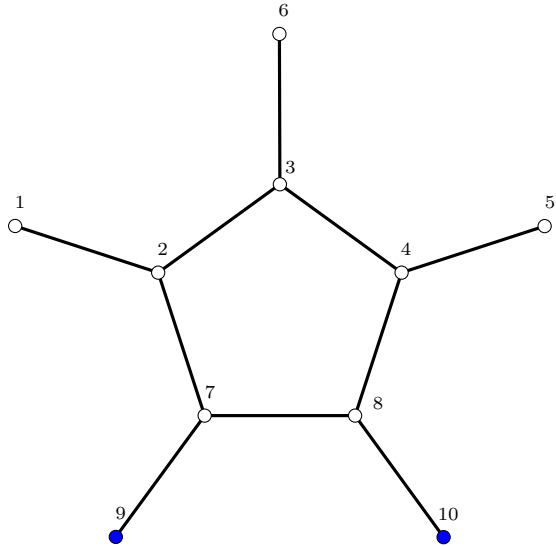

Another natural question is if there are combinatorially symmetric Markov chains which can be identified from a number of nodes smaller than their zero forcing number. This would be analogous to the relationship between zero forcing and maximum nullity (see, e.g.,~\cite{hogben}), where it is known that the forcing number is an upper bound to the nullity, but that this bound is not always tight. We conjecture that the same happens here, and a natural candidate to test would be one for which the maximum nullity is unequal to the forcing number, such as $H_5$ (compare Fig. 1), which has a forcing number of $3$ and a maximum nullity of $2$. We conjecture that the corresponding Markov chain can be uniquely identified using nodes $\{9,10\}$ and will provide some evidence towards this conjecture. Firstly, using Lemma~\ref{inductionstep} we see that we can obtain  $(P^{n})_{ik}$ for $i,k \in \{2,4,7,8,9,10\}$ and all $n$ easily. This leaves us with the $14$ unknown parameters $P_{11}, P_{12}, P_{21}, P_{23}, P_{32}, P_{33}, P_{36}, P_{63}, P_{66}, P_{34}, P_{43}, P_{45}, P_{54}, P_{55}.$ By using row-stochasticity we may eliminate $P_{12}, P_{23}, P_{36}, P_{45}, P_{55}, P_{66}$ to be left with $8$ unknowns. We claim that these can be obtained from $(P^{2})_{22}, (P^{2})_{44}, (P^{2})_{24}, (P^{2})_{42},$ $(P^{3})_{22}, (P^{3})_{44}, (P^{3})_{42},$ and $(P^{4})_{22}$. Expanding the matrix powers in terms of the matrix elements like in Equation \ref{poly}, we obtain a set of 8 polynomial equations (which are, however, too complex to be worthwhile writing down explicitly here). For example, $(P^{4})_{22}$ is linear in $P_{63}$ and the only equation involving this parameter, which leaves us with $7$ remaining equations and unknowns. $(P^{2})_{42}$ and $(P^{3})_{42}$ can be combined to obtain
\begin{equation}
	P_{33}=\left[(P^{3})_{42}-(P^{2})_{42} P_{22} -(P^{2})_{42} P_{44}-P_{48}P_{72}P_{78}\right]/(P^{2})_{42}.
\end{equation}
The equations from $(P^{2})_{22},(P^{2})_{44},(P^{2})_{24}$ and $(P^{2})_{42}$ are linear in $P_{11},P_{54},P_{34}$ and $P_{43}$ respectively and can thus eliminate those variables. This leaves us with two equations originating from $(P^{3})_{22}$ and $(P^{3})_{44}$ and two unknown $P_{32}$ and $P_{21}$. These equations are rather complex in their coefficients but quartic in the variables (and quadratic in $P_{32}$ and $P_{21}$, respectively). Usually this equation system would have $4$ solutions, while for certain degenerate cases it might have infinitely many. What makes it hard to analyse these equations is that we found examples with several feasible solutions for $P_{32}$ and $P_{21}$, which nonetheless all  lead to a unique feasible solution for $P$ (e.g, a valid stochastic matrix $P$).  While we were unable to prove in general that the whole set of equations has a unique feasible solution, we have checked $10000$ randomly created Markov chains with the graph $H_5$ and found via exact computer algebra that they had only one feasible solution. This leads us to conjecture that this chain is identifiable from two sites (at the least, the numerical analysis shows that most systems with this graph are uniquely identifiable). One may therefore speculate that the identifiability is related more closely to the maximum nullity after all.

Perhaps this discussion highlights another point- checking identifiability is a very complicated subject of algebraic geometry in general, and even if the system might be uniquely solvable, finding this solution is hard, both analytically and numerically. On the other hand, Zero Forcing is easy to check, and the inductive protocol for identification above is constructive and easy to implement.

\section{Continuous time Markov chains}
A naive way to consider the case of continuous Markov chains would be to integrate their dynamics for some time $t$ and then refer to the discrete case. However, typically these discrete Markov chains will be described by the fully connected graphs and the above protocol would require to observe $S-1$ states. It is therefore advisable to attempt the analysis with the transition rate matrix $Q$ instead. It is well known that $Q$ has non-positive off-diagonals and normalised diagonal elements $(Q)_{ii}=\sum_{j\neq i}(Q)_{ij}$. Furthermore $P(t)=\exp (Qt)$ is the transition matrix of a Markov chain for all $t\in \mathbb{R}_{\ge 0}$.

We assume that $Q$ is combinatorially symmetric and define the graph of the continuous time Markov chain following Definition \ref{graphofmatrix} as $\mathcal{G}(Q)$. Provided this graph is connected, the analogous version to Lemma \ref{observing} is that for $i,j$ observable, $P(t)_{ij}$ can be estimated arbitrarily well for all times. Via \begin{equation}
	 \left. \frac{d^n}{dt^n}P(t)_{ij}\right|_{t=0} =(Q^n)_{ij}
\end{equation}
 we can perform the same iteration as in Lemma \ref{inductionstep} along a zero forcing set, with the row stochasticity replaced by $(Q)_{ii}=\sum_{j\neq i}(Q)_{ij}$. This leads us to the following 
\begin{thm}
Consider combinatorially symmetric continuous-time Markov chain $Q$ with graph $\mathcal{G}(Q)$. If the observable set of vertices on $\mathcal{G}(Q)$ is zero forcing, then the Markov chain is uniquely determined through the dynamics of the observed states.
\end{thm}
\section*{acknowledgments}
I acknowledge the feedback of Leslie Hogben, Bryan Shader and Tracy Hall and our enjoyable discussions at the AIM workshop `Zero forcing and its applications', as well as the AIM for providing the funding to attend this workshop. This work was supported by the EPSRC Grant No. EP/M01634X/1.

\end{document}